\documentclass[12pt,a4paper]{amsart}
\usepackage[utf8]{inputenc}
\usepackage{amssymb,amsfonts,amsmath,amsthm,graphicx,url}
\usepackage{tikz,mathtools}
\usepackage{amsaddr}
\newtheorem{theorem}{Theorem}
\newtheorem{example}{Example}
\newtheorem{lemma}[theorem]{Lemma}
\newtheorem{claim}[theorem]{Claim}
\newtheorem{proposition}[theorem]{Proposition}
\newtheorem{property}[theorem]{Property}
\newtheorem{corollary}[theorem]{Corollary}

\newtheorem{conjecture}[theorem]{Conjecture}
\newtheorem{remark}[theorem]{Remark}
\newcommand{\sep}{\mathsf{sep}}

\newcommand{\Ip}{I_\mathsf{p}}

\begin{document}
\title[List Coloring with Separation of the Complete Graph]{On List Coloring with Separation of the Complete Graph and Set System Intersections}
\author{Jean-Christophe Godin}
\address{Institut de Math\'ematiques de Toulon, Universit\'e de Toulon, France}

\email{godinjeanchri@yahoo.fr}
\author{Rémi Grisot}
\address{LEAT, Université Côte d’Azur}

\email{remi.grisot@etu.univ-cotedazur.fr}
\author{Olivier Togni}
\address{LIB, Université Bourgogne Franche-Comté, France}

\email{olivier.togni@u-bourgogne.fr}

%
%


\date{}

\begin{abstract}
We consider the following list coloring with separation problem: Given a graph $G$ and integers $a,b$, find the largest integer $c$ such that for any list assignment $L$ of $G$ with $|L(v)|= a$ for any vertex $v$ and $|L(u)\cap L(v)|\le c$ for any edge $uv$ of $G$, there exists an assignment $\varphi$ of sets of integers to the vertices of $G$ such that $\varphi(u)\subset L(u)$ and $|\varphi(v)|=b$ for any vertex $u$ and $\varphi(u)\cap \varphi(v)=\emptyset$ for any edge $uv$. Such a value of $c$ is called the separation number of $(G,a,b)$.  Using a special partition of a set of lists for which we obtain an improved version of Poincaré's crible, we determine the separation number of the complete graph $K_n$ for some values of $a,b$ and $n$, and prove bounds for the remaining values.
\end{abstract}

\maketitle

\section{Introduction}

Let $a,b,c$ be integers and let $G$ be a graph. A $a$-list assignment $L$ of $G$ is a function which associates to each vertex a set of $a$ integers. The list assignment $L$ is {\em $c$-separating} if for any $uv\in E(G)$, $|L(u)\cap L(v)|\le c$. The graph $G$ is {\em $(a,b,c)$-choosable} if for any $c$-separating $a$-list assignment $L$, there exists an $(L,b)$-coloring of $G$, i.e. a coloring function $\varphi$ on the vertices of $G$ that assigns to each vertex $v$ a subset of $b$ elements from $L(v)$ in such a way that $\varphi(u)\cap \varphi(v)=\emptyset$ for any $uv\in E(G)$.

The list coloring problem with restrictions on the list intersections has been introduced by Kratochvíl, Tuza and Voigt~\cite{KTV98a}. Notice that Kratochvíl et al.~\cite{KTV98a, KTV98} defined $(a,b,c)$-choosability a bit differently, requiring for a $c$-separating $a$-list assignment $L$ that the lists of two adjacent vertices $u$ and $v$ satisfy $|L(u)\cap L(v)|\le a-c$. Among the first results on the topic, a complexity dichotomy was presented~\cite{KTV98a} and general properties given~\cite{KTV98}.  These first papers were followed by a series of papers considering choosability with separation of planar graphs, mainly for the case $b=1$~\cite{BCD+, CLW18, CYR+, CLS16, CFWW, KL15, Skr01}. While the fact that planar graphs are $(4,1,2)$-choosable was proved very recently~\cite{Zhu22}, a still open question is whether all planar graphs are $(3,1,1)$-choosable or not. Other recent papers concern balanced complete multipartite graphs and $k$-uniform hypergraphs (for the case $b=1$)~\cite{FKK}; bipartite graphs (for the case $b=c=1$)~\cite{EKT19}; a study with an extended separation condition~\cite{KMS}, and cycles and outerplanar graphs for arbitrary $b$~\cite{GT21}.

In this paper, we concentrate on choosability with separation of complete graphs. As a $(a,b,c)$-choosable graph is also $(a,b,c')$-choosable for any $c'<c$, our aim is to determine, for given $a,b$, $a\ge b$, the largest $c$ such that $G$ is $(a,b,c)$-choosable. Following our previous work on cycles~\cite{GT21}, we define the {\em  (list) separation number} $\sep(G,a,b)$ of $G$ as
\[\sep(G,a,b)=\max\{c, G \text{ is } (a,b,c)\text{-choosable}\}.\]

Notice that we have $0\le \sep(G,a,b)\le a$ for any graph $G$ and $a\ge b$, hence this parameter is well defined.

In our setting, we know that any planar graph $G$ satisfies $\sep(G,5,1)=5$ (Thomassen's Theorem), $\sep(G,4,1)\ge 2$~\cite{Zhu22} ; but we do not know if $\sep(G,3,1)\ge1$ holds for all planar graphs $G$. For the complete graph, Kratochvíl et al.~\cite{KTV98} proved that $\sep(K_n,\lfloor\sqrt{n-11/4}+3/2\rfloor,1)\ge 1$. Moreover, the separation number of the cycle is determined and bounds are given for catuses and outerplanar graphs~\cite{GT21}.

The following Hall-type condition that we call the {\em amplitude condition} is necessary for a graph $G$ to be $(L,b)$-colorable: 
$$\forall H\subset G, \sum_{k\in C} \alpha(H,L,k) \ge b|V(H)|,$$ where $C=\bigcup_{v\in V(H)}L(v)$ and $\alpha(H,L,k)$ is the independence number of the subgraph of $H$ induced by the vertices containing $k$ in their color list. Notice that $H$ can be restricted to be a connected induced subgraph of $G$. As shown by Cropper et al.~\cite{CGHHJ} (in the more general context of weighted list coloring), this condition is also sufficient when the graph is a complete graph or a path (or some other specific graphs).

For a list assignment $L$ on a graph $G$ of order $n$ with vertex set $V(G)=\{v_1,v_2, \ldots, v_n\}$ and for $S\subset[1,n]$, we write $\Sigma_{S}(L)=\sum_{k\in C}\alpha(H,L,k)$, where $H$ is the subgraph of $G$ induced by  $\{v_i, i\in S\}$. 

Remark that if $G$ is a complete graph, then $\alpha(H,L,k) = 1$ for any $k$. Hence the amplitude condition for $K_n$ becomes \begin{equation}\label{eq1}
\forall S\subset[1,n], \Sigma_{S}(L)=|\cup_{i\in S} L(x_i)|\ge b|S|                                                                                                                     
\end{equation}

\begin{example}\label{ex1}
For the complete graph $K_4$ with vertex set $\{v_1,v_2,v_3,v_4\}$, let $L$ be the 3-separating $5$-list assignment defined by:\\
$L(v_1)=\{1,2,3,4,5\}$, $L(v_2)=\{1,2,3,6,7\}$, $L(v_3)=\{3,4,6,7,8\}$, and $L(v_4)=\{4,6,8,9,10\}$.

We have $L(v_1)\cup L(v_2)\cup L(v_3) =\{1,2,3,4,5,6,7,8\}$, hence $\Sigma_{\{1,2,3\}}(L) = |\{1,2,3,4,5,6,7,8\}|=8$. Therefore $K_4$ is not $(L,3)$-colorable since $\Sigma_{\{1,2,3\}}(L)=8<3b=9$. (Note that we also have $\Sigma_{[1,4]}(L)=10 <4b=12$ in this case.)
\end{example}

For the separation number of the complete graph, the following properties are easy to prove:
\begin{property}\label{propp1}Let $a,b,n$ be integers. Then
\begin{itemize}
\item for fixed $b,n$, the function $\sep(K_n,a,b)$ is increasing with $a$;
\item for fixed $a,b$, the function $\sep(K_n,a,b)$ is decreasing with $n$.
\end{itemize}
\end{property}

Moreover, we observed (and will prove it in the case $b\le a< 2b$ and $(n-1)b \le a < nb$) that for any $a,b,n$, $\sep(K_n,a+1,b)\le \sep(K_n,a,b)+2$.

In Section 2, we introduce proper intersections of set systems and show some of their properties and use them to partition the lists of colors of the vertices of $K_n$, allowing to simplify the computations. In Section 3, we propose a general coloring algorithm and two special types of list assignments with nice properties that will be used mainly for finding good counter-examples. Then in Section 4, we determine bounds and exact values for the separation number of the complete graph $K_n$, depending on $a,b$ and $n$, and finnish with a conjecture. The algebra that allowed us to find the counter examples of Section 4 is given in Appendix~\ref{app}.


\section{Algebraic preliminaries on set systems intersections}

\subsection{Proper intersections and Poincaré's crible improvment}

Let $n\ge 1$ be an integer and $[n]=[1,n]$. For sets of elements $A_1, A_2, \ldots, A_n$, we define the following:
For any $i\ge 1$ and $S_i=\{\alpha_1,\alpha_2, \ldots, \alpha_i\}\subset [n]$, the {\em proper intersection} $\Ip(S_i)$ of $A_{\alpha_1}, A_{\alpha_2}, \ldots , A_{\alpha_i}$ is given by 
 \[\Ip(S_i)=\left(\bigcap_{k=1}^i A_{\alpha_k}\right)\setminus \left(\bigcup_{\beta\in [n] \setminus S_i} A_{\beta}\right).\]
 
Notice that for any set $S\in[n]$, the whole $n$ sets $A_i$ are present in the formulae of the proper intersection of $S$. For example, for $n=5$,
$\Ip(\{1\})= A_1\setminus (A_2\cup A_3\cup A_4\cup A_5)$ ; $\Ip(\{2,3\})=(A_2\cap A_3)\setminus (A_1\cup A_4\cup A_5)$; $\Ip(\{1,2,5\})=(A_1\cap A_2\cap A_5)\setminus (A_3\cup A_4)$; $\Ip(\{2,3,4,5\})=(A_2\cap A_3\cap A_4 \cap A_5)\setminus A_1$ and $\Ip(\{1,2,3,4,5\})= A_1\cap A_2\cap A_3\cap A_4 \cap A_5$.

Then the classical intersection can be described in terms of proper intersections:
\begin{property}\label{prop1} For any $S\subset[n]$, 
 \[\bigcap_{i\in S} A_{i} = \bigcup_{S\subset S'\subset [n]} I_p(S').\]
\end{property}
\begin{proof}
The proof is by double inclusion. First, by construction, For any $S'\supset S$, any $x$ in  $I_p(S')$ is also an element of $\cap_{i\in S} A_{i}$. Next, if $x\in \cap_{i\in S} A_{i}$, then either $x\in I_p(S)$ and we are done or $x\not\in I_p(S)$. In this latter case, let $\overline{S}$ be the set of indices $j$ of $[n]\setminus S$ for which $x\in A_{j}$. Then, by construction, we have $x\in I_p(S')$, for $S'=S\cup \overline{S}$.
\end{proof}

Thanks to this property, we obtain the fact that the proper intersections form a partition of $\cup_{i=1}^{n} A_i$.
\begin{property}\label{propPart}
For any $S,S'\subset [n]$ such that $S\ne S'$, we have 
\[I_p(S)\cap I_p(S') = \emptyset.\]
\end{property}
\begin{proof}
At least one of $S-S'$ or $S'-S$ is non empty. Let us say $S-S'$ is not empty and let $s\in S-S'$. Let $x\in I_p(S)\cap I_p(S')$. Then $x\in A_s$, but $x\in I_p(S')$ and $s\not\in S'$. Hence $x\not\in A_s$, a contradiction and therefore the intersection is empty.
\end{proof}

We thus have the following {\em à-la-Poincaré} result:
\begin{theorem}\label{thm1}
 \[\bigcup_{i=1}^{n} A_i = \bigcup_{S\subset [n], S\ne \emptyset} I_p(S).\]
\end{theorem}

\begin{proof}
 By Property~\ref{prop1} for $|S|=1$, we have $A_{i}= \bigcup_{S'\subset [n]\setminus \{i\}} I_p(\{i\}\cup S')$.
 Therefore $\bigcup_{i=1}^n A_i = \bigcup_{i=1}^n \bigcup_{S'\subset [1,n]\setminus \{i\}} I_p(\{i\}\cup S') = \bigcup_{S\subset [n], S\ne \emptyset} I_p(S)$.
\end{proof}

By Property~\ref{propPart} and Theorem~\ref{thm1}, we obtain the following crible which can be seen as an improvment of Poincaré's crible as we only do additions (compared to Poincaré's crible where additions and substractions alternate):
\begin{corollary}~\label{corAmp}
 \[\left|\bigcup_{i=1}^{n} A_i\right| =\sum_{S\subset [n], S\ne \emptyset} |I_p(S)|. \]
 
\end{corollary}


\subsection{List assignments of $K_n$ and proper intersections}

We will use proper intersections on list assignments of $K_n$.
Let $V(K_n)=[n]$. 
For a list assignment $L$ of $K_n$, we let $A_i=L(i)$, for $1\le i\le n$.

In a basic way, $L$ is a set of $n$ lists of size $a$. Using proper intersection tools from the previous subsection, the list assignment $L$ can now be represented by the vector of proper intersections cardinals $\mathbf{V}(L)$ of dimension $2^n-1$ (ordered by set inclusion and alphanumeric order). 

Then we have \begin{equation} \label{eq1p}
\forall i\in[n]\sum_{S\subset [n], i\in S} |\Ip(S)| = |L(i)|.
\end{equation}

And the cardinality of the intersection between two lists is given by 
\begin{equation} \label{eq2}
 \forall i,j\in [n], i\ne j, \sum_{S\subset [n], i,j \in S}|I_p(S)| =|L(i)\cap L(j)|.
\end{equation}

Moreover, thanks to Corollary~\ref{corAmp}, the total amplitude of $L$ is given by \begin{equation} \label{eq3p}\Sigma_{[n]}(L)=\sum_{S\subset [n]} |\Ip(S)|.\end{equation}

\begin{example}\label{ex2}
For the complete graph $K_4$ and the list assignment $L$ defined in Example~\ref{ex1}, we have :
$\Ip(\{1\})=\{5\}$, $\Ip(\{4\})=\{9,10\}$, $\Ip(\{1,2\})=\{1,2\}$, $\Ip(\{2,3\})=\{7\}$, $\Ip(\{3,4\})=\{8\}$, $\Ip(\{1,2,3\})=\{3\}$,  $\Ip(\{1,3,4\})=\{4\}$, and  $\Ip(\{2,3,4\})=\{6\}$. Hence the vector $\mathbf{V}(L)=(1,0,0,2,2,0,0,1,0,1,1,0,1,1,0)$.
\end{example}

Consequently, from these equations, we obtain an ILP-formulation of the problem of finding smallest counter examples ($c$-separating $a$-list assignments $L$ of $K_n$ for which no $(L,b)$-coloring exists): For $S\subset [n]$, we consider the variable $x_S=|\Ip(S)|$. Then the goal is to minimize $c$ subject to the constraints:
\begin{equation} \label{eq=a}\forall i\in[n], \sum_{S\subset [n], i\in S} x_S = a,
\end{equation}
\begin{equation} \label{eq<c}
\forall i,j\in [n], i\ne j, \sum_{S\subset [n], i,j\in S} x_S \le c,
\end{equation}
and
\[\sum_{S\subset [n]} x_S <nb.\]

This formulation allows to use ILP-solvers to help us finding counter-examples.

\begin{remark}
Assuming that the colors are taken from the set $\{1,\ldots,na\}$ (this is always possible, up to a color renumbering), we can count the number of different list assignments associated to the same vector of proper intersections $\mathbf{V}(L)$: Considering any ordering $S_1$, $S_2,\ldots, S_{2^n-1}$ of the non-empty subsets of $[n]$, this number of list assignments is given by $$\prod_{i=1}^{2^n-1}\binom{na-\sum_{j<i}|\Ip(S_j)|}{|\Ip(S_j)|}=\frac{(na)!}{(na-\sum_{i}|\Ip(S_i)|)!(\prod_{i} |\Ip(S_i)|!)}.$$

For instance, if all proper intersections are equal to zero execpt $\Ip([n])=a$, then we obtain $\binom{na}{a}$ lists.
\end{remark}

\subsection{Counting list assigments up to proper intersection equivalence}
In order to compute the gain of working with proper intersections instead of list assignments, we are now going to count the total number of $a$-list assignments up to proper intersection equivalence. Let $\mathcal{L}(n,a)$ be the set of all $a$-list assignments on $n$ vertices up to proper intersection. The cardinal of $\mathcal{L}(n,a)$ can be compared with the total number of $a$-list assignments on $na$ colors, which is ${na \choose a}^n$.


For $n=2$, it is easy to see that there are only $a+1$ different lists up to proper intersections, hence $|\mathcal{L}(2,a)|=a+1$.
We define $\mathcal{L}(n,a, [n]=0)$ as the subset of $\mathcal{L}(n,a)$ for which $\Ip([n])=\emptyset$ and $\mathcal{L}(n,a, [n]>0)$ as the subset of $\mathcal{L}(n,a)$ for which $|\Ip([n])|>0$.

\begin{lemma}\label{Le5}
For integers $a\ge 1,n\ge 1$, we have : \[|\mathcal{L}(n,a)|=|\mathcal{L}(n,a-1)|+ |\mathcal{L}(n,a, [n]=0)|.\]

\end{lemma}

\begin{proof}
It is easily seen that $\mathcal{L}(n,a) = \mathcal{L}(n,a, [n]>0) \cup \mathcal{L}(n,a, [n]=0)$. Now, note that there is a trivial isomorphism between lists of 
$\mathcal{L}(n,a, [n]>0)$ and that of $\mathcal{L}(n,a-1)$. To see this, represent each list by the vector of the cardinals of their proper intersections, ordered by the size. Then since the last element of each vector is positive, its value can be decreased by one, obtaining a vector corresponding with a list assignment of  $\mathcal{L}(n,a-1)$.
\end{proof}

Remark that any element of $\mathcal{L}(n,a)$ must satisfy Equation~\ref{eq=a} for each vertex and reciprocally any solution of this set of $n$ equations gives a list assigment $\mathcal{L}(n,a)$. Hence the set of all solutions is a subspace of dimension $2^n -1 -n$ of a vectorial space of this equation set.

\begin{lemma} For all $S=\{a_1,a_2,\ldots,a_i\}\subset[n], |S|\ge 2$, the vectors $\mathbf{v}_S =\mathbf{e}_S -\sum_{j=1}^i \mathbf{e}_{a_j} $ form a canonical base of $\mathcal{L}(n,a)$.
\end{lemma}
\begin{proof}
As we have $2^n -1$ proper intersections, then the space is of dimension $2^n -1$. But we have $n$ linear equations (Equation~\ref{eq=a}) and since they are linearly independent, thus $\mathcal{L}(n,a)$ is of dimension $2^n -1-n$. It then suffices to prove that the family of $\mathbf{v}_S$ is free. For any $S\subset[n], |S|\ge 2$, let $\alpha_S$, such that $\sum\alpha_S\mathbf{v}_S=\mathbf{0}$. We have $\sum\alpha_S\mathbf{v}_S=\sum_{i=1}^n\beta_i\mathbf{e}_{i} + \sum_{|S|\ge 2}\alpha_S\mathbf{e}_S=\mathbf{0}$. As the $\mathbf{e}_S$ are vectors of the canonical base, then $\alpha_S=0$ for every $S$. Hence the family is free.
\end{proof}

There exists a trivial solution to the set of equations, it is the list assignment with vector $\mathbf{T}=(a,a,\ldots,a,0,\ldots,0)$ with $n$ values equal to $a$. Then any list assignment solution can be written as the trivial solution plus a linear combination of these base vectors $\mathbf{e}_S$. Hence a list assignment solution can be identified by a linear sum of base vectors. We therefore only have to work on linear sums of base vectors. But, as the entries of a solution have to be non negative, we only have a finite number of solutions.
We define the set of list assignments 
$\mathcal{L}(n,=a)$ 
as the set of list assignments such that there exists some $i\in[n]$ for which the sum of coefficients $\alpha_S$ over all $S$ containing  $i$ in the linear combination is equal to $a$ in the vectorial representation (i.e., such that $\Ip(\{i\})=\emptyset$) and for which $\Ip([n])=\emptyset$.
 
\begin{lemma}\label{Le7}
\[|\mathcal{L}(n,a, [n]=0)|=|\mathcal{L}(n,a-1, [n]=0)| + |\mathcal{L}(n,=a)|.\]
\end{lemma}
\begin{proof}
 
The set $\mathcal{L}(n,a, [n]=0)$ trivially decomposes into the set $\mathcal{A}$ of list assignments for which all the first $n$ entries of the vector are strictly positive and the set of which there is at least a zero in the first $n$ entries of the vector, i.e., $\mathcal{L}(n,=a)$. But there is a natural bijection between $\mathcal{A}$ and $\mathcal{L}(n,a-1, [n]=0)$, hence the recursion formulae.
\end{proof}

\begin{lemma}\label{Le8}
\[|\mathcal{L}(3,=a)| = \left\{\begin{array}{ll}
3/4a^2+3/2a+1 & a\ \mathrm{even}\\
3/4a^2+3/2a+3/4 & a\ \mathrm{odd}.
\end{array}\right.
\]
\end{lemma}
\begin{proof}
For $n=3$, we have only 3 vectors: $S=(1,2), (1,3)$, and $(2,3)$. Hence any solution can be represented by a triplet of vectors $(\alpha_1,\alpha_2,\alpha_3)$. Starting from $(a,0,0)$, we count the triplets by equivalence classes such that the first coefficient is maximum, the second one is greater than or equal to the third one and then we count the number of elements in the class of equivalence. We then have three cases: if the three numbers are different then there are $3!=6$ different combinations; if exactly two numbers are different then we have $3$ combinations and only one if the three numbers are all identical. Hence each solution is of the form $(a-x,x,y)$ with $0\le x\le a/2$ and $y\le x$. Therefore, we obtain the following formulae:

\[|\mathcal{L}(3,=a)| = \left\{\begin{array}{ll}
1 + 3a/2 +\sum_{i=a/2+1}^a(3+6(a-i)), & a\ \mathrm{even}\\
\sum_{i=(a+1)/2}^a(3+6(a-i)), & a\ \mathrm{odd}.
\end{array}\right.
\]

Hence, using simple calculations, we obtain the formulae of the lemma.

\end{proof}

\begin{proposition}\label{propL3}
For $n=3$, we have \[|\mathcal{L}(n,a)|=\frac{1}{16}(a^4+8a^3+24a^2+32a+16-\epsilon(a)),\] where $\epsilon(a) = 1 $ if $a$ is odd and $\epsilon(a)=0$ otherwise.
\end{proposition}
\begin{proof}
Combining the two recursion formulas of Lemmas~\ref{Le5} and \ref{Le7}, we obtain the formulae for the residue, allowing to find the general polynomial.
\end{proof}

\begin{conjecture}\label{conj1}
$|\mathcal{L}(n,a)|$ is a polynomial in $a$ of degree $n+1$, with coefficients being functions of Bernouilli's numbers.
\end{conjecture}

Proposition~\ref{propL3} (and Conjecture~\ref{conj1}, if true) says that using proper intersections allows to go from an exponential number of list assignments to treat to a polynomial number.

%

\section{$(a,b,c)$-choosability of $K_n$}

\subsection{General coloring algorithm}
We provide an algorithm that, given a list assigment $L$ of $K_n$, produces a multi-coloring. We conjecture that this algorithm is optimal (i.e., produces an $(L,b)$-coloring if $K_n$ is $(L,b)$-colorable).

\paragraph{Algorithm ColorSym:} Taking an arbitrary list-assignment $L$ as input, we consider at each step the vector $\mathbf{w}=(w_1,\ldots,w_n)$, with $w_i=$ being the number of colors given to $v_i$ so far.

Step 1. Color every proper intersections of sets of size 1 : $v_i$ gets $|I_p(\{i\})|$ colors. Thus $w_i=|I_p(\{i\})|$.


Step $i\ge 2$. If $w_i\ge b$ for each $i$ or $i=n+1$, then Stop. Otherwise, for each set $S$ with $|S|=i$, consider $|I_p(S)|$ and let $j\in S$ be the index for which $w(j)$ is minimum (if more than one $j$ with minimum $w$, then take the smallest index). Remove a color from $I_p(S)$ and assign it to $v_j$. Goto Step $i$ until $\Ip(S)$ is empty.

\subsection{Symetrical list assignments on $K_n$}
Let $L$ be a $a$-list assignment on $K_n$.
We say that $L$ is {\em symetrical} if for any $i$, $1\le i \le n$, and any $S,S'\subset [n]$ such that $|S|=|S'|$, we have $|I_p(S)|=|I_p(S')|$. For such a list assignment, we let $x_i = |I_p(S)|$ for $S\subset [n], |S|=i$ and $\mathbf{x}(L)=(x_1,x_2,\ldots, x_n)$.

Considering symetrical list assignments allows to reduce the number of variables in the linear program from $2^n-1$ to $n$ and to simplify Equation~\ref{eq=a} into 
\begin{equation} \label{eq1s=a}
 \sum_{i=1}^{n} \dbinom{n-1}{i-1} x_i = a,
\end{equation}
and Equation~\ref{eq<c} into 
\begin{equation} \label{eq1s<c}
 \sum_{i=2}^{n} \dbinom{n-2}{i-2} x_i \le c,
\end{equation}
and for the total amplitude: 
\begin{equation} \label{eq2s>=nb}
 \sum_{i=1}^{n} \dbinom{n}{i} x_i \ge nb.
\end{equation}

As we will see, symetrical list assignments have the nice property that Equation~\ref{eq2s>=nb} is sufficient to guarantee an $(L,b)$-coloring of $K_n$. We will first prove how to find a balanced coloring of a set of subsets of $[n]$.
We define $\mathcal{P}(i,n)$ as the set of all subsets of cardinality $i$ of the set $\{1,\ldots, n\}$.
Observe that $|\mathcal{P}(i,n)|=\binom{n}{i}$. Let also $\mathcal{P}(n) = \cup_{i=1}^n \mathcal{P}(i,n)$.
For a weight vector $\mathbf{w}=(w_1,\ldots, w_n)$, with $w_i\ge 0$ for all $i$, $1\le i\le n$, let $\mathcal{P}(n)^\mathbf{w}$ be the multiset obtained from the empty set by adding to it $w_i$ times the set $\mathcal{P}(i,n)$ for any $i$, $1\le i\le n$.

A partition (or coloring) $P_1,\ldots, P_n$ of a subset $S$ of $\mathcal{P}(n)^\mathbf{w}$ is {\em balanced} if for any $i,j$, $1\le i\ne j\le n$,  $||P_i|-|P_j||\le 1$ and for any $X \in P_j$, $j\in X$.


\begin{lemma}\label{lemmaBalance}
For any integer $n\ge 3$ and any wheight vector $\mathbf{w}$ of size $n$, there exists a balanced partition of $\mathcal{P}(n)^\mathbf{w}$.
\end{lemma}

\begin{proof}
Observe that $\mathcal{P}(n)^\mathbf{w}$ is the union of sets $\mathcal{P}(i,n)$. 
Let $S_n$ be the group of permutations of $n$ elements. Remark that $S_n$ acts on $\mathcal{P}(i,n)$. Let $g$ be the $n$-elements cycle. Hence $\mathcal{P}(i,n)$ can be viewed as the set of its orbits by $g$ and $\mathcal{P}(n)^\mathbf{w}$ as a multiset of orbits.

The general idea to prove the lemma is to color $\mathcal{P}(n)^\mathbf{w}$ by coloring each orbit of each set $\mathcal{P}(i,n)$ one after the other in any order while ensuring that at each step, the partial coloring is balanced. The key idea for showing this is possible is the following observation:

\begin{claim}\label{claim1}
For any integers $i$ and $j$, $1\le i,j\le n$, any orbit $O$ of $\mathcal{P}(i,n)$ can be colored with colors $j,j+1, \ldots, j+ |O|$ (with $n+1=1$).
\end{claim}
\noindent\textit{proof.}
Observe first that all elements of $[n]$ are present at least once in $O$. Hence, there exists $o_j\in O$ such that $o_j$ contains $j$. Color $o_j$ by color $j$ and color each remaining element $g^k(o_j)$ of $O$ with color $j+k$ for $k\in \{1,\ldots, n-1\}$ ($j+k-n$ if $j+k>n$).

The algorithm for obtaining a balanced coloring $\varphi$ consists in coloring the orbits of $\mathcal{P}(n)^\mathbf{w}$ in any order and, when coloring the elements of a new orbit, to start with the color $j$ such that $|\varphi^{-1}(j)| < |\varphi^{-1}(j-1)|$ if such a $j$ exists and $j=1$ otherwise and use Claim~\ref{claim1}. After each such step, we can observe that the (partial) coloring $\varphi$ remains balanced. Therefore, after having colored the last element of the last orbit, $\varphi$ is a balanced coloring of $\mathcal{P}(n)^\mathbf{w}$.

\end{proof}

\begin{corollary}
Let $L$ be a symetrical list assignment of $K_n$. If $K_n$ is $(L,b)$-colorable, then Algorithm ColorSym produces an $(L,b)$-coloring.
\end{corollary}

\begin{proof}
To any symetrical list assignment $L$, we can make correspond a set $\mathcal{S}\subset\mathcal{P}(n)^\mathbf{x}$ (each set $S\subset[n]$, $|S|=i$, is present $x_i$ times in $\mathcal{S}$). Hence by Lemma~\ref{lemmaBalance}, there exists a balanced coloring of $\mathcal{S}$, and thus a coloring $\varphi$ from $L$ which is balanced, i.e., for any $i,j\in [n]$,  $||\varphi(v_i)|-|\varphi(v_j)||\le 1$. Therefore the amplitude condition has only to be verified globally, i.e., it suffices to show that Equation~\ref{eq2s>=nb} is satisfied in order $K_n$ to be $(L,b)$-colorable. As Equation~\ref{eq2s>=nb} is a necessary condition, if an $(L,b)$-coloring of $K_n$ exists then Algorithm ColorSym ends up with such a coloring.
\end{proof}

\subsection{Quasi-symetrical list assignments}
Let $L$ be a $a$-list assignment on $K_n$.
We say that $L$ is {\em quasi-symetrical} if for any $i$, $1\le i \le n$, and any $S,S'\subset [n]$ such that $|S|=|S'|$, we have $-1\le  |I_p(S)|-|I_p(S')| \le 1$. 

The {\em vector of cardinality-support} $\mathbf{cs}(S)$ of a set $S=\{S_1,S_2,\ldots, S_k\}$ of subsets of $[n]$ is defined by $\mathbf{cs}(S)=(\beta_1,\ldots,\beta_n)$, with $\beta_i=|\{i\in S_j, j\in \{1,\ldots,k\}\}|$.

\begin{lemma}\label{le:xx+1}
For any $n,k$ with $n\ge 4$, $k\ge 2$ and $n\ge (2k-1+\sqrt{8k+1})/2$, there exist a set of $k+1$ $k$-subsets $\{S'_1,S'_2,\ldots, S'_{k+1}\}$ of $[n]$ and a set of $k$ $(k+1)$-subsets $\{S_1,S_2,\ldots, S_k\}$ of $[n]$ having the same cardinality-support vector.
\end{lemma}

\begin{proof}
Let $n,k$ be as in the statment of the Lemma. We consider two cases : $2\le k \le n/2$ and $n/2< k \le n-2$.

\textbf{ Case 1 :} $2\le k \le n/2$. For $1\le i\le k$, set $S'_i=\{1,\ldots,k,k+i\}$

For $1\le j\le k$, set $S_j=S'_j\setminus\{j\}$ and $S_{k+1}=\{1,\ldots,k\}$.

By construction, they have the same cardinality support vector that starts with $k$ times the entry $k$ and $k$ times the entry $1$, i.e., $\mathbf{cs}(S)=(k,k,...,k,1,...,1,0, ...,0)$. 

\textbf{ Case 2 :} $k\ge n/2$. 
For $1\le i\le k$, set $S'_i=\{1,\ldots,k-1,x_i,y_i\}$, with $\{x_i,y_i\}\subset [k,n]$ and $x_i\ne y_i$. We choose for each $S'_i$ a different couple. This is always possible since there are $\dbinom{n-k+1}{2}$ such couples and we have $\dbinom{n-k+1}{2}\ge k$ since $n\ge (2k-1+\sqrt{8k+1})/2$.

For $1\le j\le k-1$, $S_j=S'_j\setminus\{j\}$ and $S_{k}=S'_k\setminus \{x_k\}$ and $S_{k+1}=\{1,\ldots,k-1, x_k\}$.
Again, by construction, both families of sets have the same cardinality support vector.

\end{proof}

\begin{proposition}\label{propapb} Let $a,b,n,x$ be integers such that $x\le n$ and $a=xb$ is a multiple of $\binom{n-1}{x-1}$.
Then $K_n$ is not $(a,b,a\frac{x-1}{n-1}+1)$-colorable.
\end{proposition}

\begin{proof}
Let $a=p\binom{n-1}{x-1}$ for some $p\ge 1$.
We first construct a symetrical list assignment $L$ by setting $|\Ip(S_x)|=p$ for every set $S_x\subset [n]$ with $|S_x|=x$ and $|\Ip(S)|=0$ for $|S|\ne x$.

We check that for every $j\in [n]$, $|L(j)|=a$ and that $\sum(L)=nb$. In fact, there are $\dbinom{n}{x}$ $x$-subsets of $[n]$ and $\dbinom{n-1}{x-1}$ of them containing some $j\in[n]$. Hence $|L(j)|=p\dbinom{n-1}{x-1}=a$ and $\Sigma(L)=p\dbinom{n}{x}=p\dbinom{n-1}{x-1}\frac{n}{x}=\frac{an}{x}=bn$. Moreover, observe that for any $i,j\in [n]$, we have $|L(i)\cap L(j)|=p\dbinom{n-2}{x-2}=p\dbinom{n-1}{x-1}\frac{x-1}{n-1}=a\frac{x-1}{n-1}$.

Now, by Lemma~\ref{le:xx+1}, there exists a set of $x+1$ $x$-subsets of $[n]$ and a set of $x$ $(x+1)$-subsets of $[n]$ with the same cardinality support vector $\mathbf{cs}$. Let $\mathcal{S}$ and $\mathcal{S'}$, respectively, be such sets as defined in the proof of Lemma~\ref{le:xx+1}. Hence, we can decrease $|\Ip(S)|$  by one for every $S\in\mathcal{S}$ (hence $|\Ip(S)|=p-1$ for such set) and increase $|\Ip(S')|$ by one for every $S'\in \mathcal{S'}$ (hence $|\Ip(S')|=1$ for such set). Let $L'$ be the quasi-symetrical resulting list assignment. Since the cardinality support vector is the same for $\mathcal{S}$ and $\mathcal{S'}$, $L'$ is a $a$-list assignment. We now prove that the amplitude of $L'$ is $\sum(L')=nb-1$ and that $L'$ is $c'$-separating, with $c'=a\frac{x-1}{n-1}+1$. For the amplitude, we have $\Sigma(L')=p\dbinom{n}{x}-(x+1)+x=p\dbinom{n}{x}-1=\Sigma(L)-1=nb-1$.
For the separation condition, if $i\in [n]$ such that $\mathbf{cs}(i)=0$, then we have, for any $j\in [n]$, $|L'(i)\cap L'(j)|= |L(i)\cap L(j)|\le c$. Otherwise, if $i,j\in [n]$ are such that $\mathbf{cs}(i)>0$ and $\mathbf{cs}(j)>0$, then we are going to prove that $|L'(i)\cap L'(j)|\le |L(i)\cap L(j)|+1\le c+1$. First, consider the case $2\le x\le n/2$. If $1\le i,j,\le x$, then, by the construction of the proof of Lemma~\ref{le:xx+1}, there are $x-2+1=x-1$ sets in $\mathcal{S}$ and $x$ sets in $\mathcal{S'}$ containing $i,j$, hence  $|L'(i)\cap L'(j)|= |L(i)\cap L(j)| -(x-1)+x=|L(i)\cap L(j)|+1 = c'$. Otherwise ($i> x$ or $j > x$), then $|L'(i)\cap L'(j)|= |L(i)\cap L(j)|$.
Second, consider the case $x>n/2$.  If $1\le i,j,\le x-1$, then, by the construction of the proof of Lemma~\ref{le:xx+1}, there are $x-2+1=x-1$ sets in $\mathcal{S}$ and $x$ sets in $\mathcal{S'}$ containing $i,j$, hence  $|L'(i)\cap L'(j)|= |L(i)\cap L(j)| -(x-1)+x=|L(i)\cap L(j)|+1 = c'$. Otherwise, if $1\le i\le x-1$ and $j\ge x$, then there is one more set in $\mathcal{S'}$ than in $\mathcal{S}$ containing both $i$ and $j$, hence $|L'(i)\cap L'(j)|= |L(i)\cap L(j)| +1=c'$. Otherwise ($i\ge x$ and $j \ge x$), we have $|L'(i)\cap L'(j)|= |L(i)\cap L(j)|$.

In conclusion, we have built a $c'$-separating $a$-list assignment $L'$ for which the amplitude condition is not fulfilled and thus $\sep(K_n,a,b)<c'$.
\end{proof}

\begin{proposition}\label{biKn}
For integers $a,b,n$ such that $a\ge 2b$ and $a$ is a multiple of $\lfloor \frac{a^2}{2b(n-1)}\rfloor$,  $K_n$ is $(a,b,\lfloor \frac{a^2}{2b(n-1)}\rfloor)$-choosable.
\end{proposition}

\begin{proof}
Let $c=\lfloor \frac{a^2}{2b(n-1)}\rfloor$ and $L$ be a $c$-separating $a$-list assignment with $a=\lambda c$ for some $\lambda\ge 1$.
Let $S$ be an $i$-subset of $[n]$. We are going to show that the amplitude condition is satisfied for $S$. We consider two cases.

If $i\le \lambda$, then since $L$ is $c$-separating, the amplitude satisfies 
\[\Sigma_{S}(L)\ge a + (a-c)+\ldots + (a-(i-1)c) = ia- \frac12 i(i-1)c.\]
Since $ic\le \lambda c\le a$ and $a\ge 2b$, we obtain $\Sigma_{S}(L)\ge ia- \frac12 (i-1)a = \frac12 (i+1)a \ge ib$.

If  $i> \lambda$, then the amplitude satisfies 
\[\Sigma_{S}(L)\ge a + (a-c)+\ldots + (a-\lambda c) = (\lambda+1)a- \frac12 \lambda(\lambda+1)c = (\lambda+1)a-\frac12 \lambda a = \frac12 (\lambda+1)a.\]

Since $c=\lfloor \frac{a^2}{2b(n-1)}\rfloor\le \frac{a^2}{2b(n-1)}$, then $\lambda = \frac ac \ge \frac{2b(n-1)}{a}$.

The above inequality with $a\ge 2b$ induces  \[\Sigma_{S}(L)\ge\frac12 (\lambda+1)a\ge  \frac{2b(n-1)+a}{2a}a = bn +\frac{a-2b}{2}\ge bn.\]

Therefore the amplitude condition is satisfied and thus $K_n$ is $(L,b)$-colorable. 
\end{proof}

\section{Separation number of $K_n$}
\label{sec:sepK4}

For $K_2$, it is easily seen that $\sep(K_2,a,b) = a-b$ if $a\le 2b$ and $\sep(K_2,a,b) = a$ if $a\ge 2b$.

For $K_3=C_3$, the separation number follows from results about the cycle from~\cite{GT21}:

\begin{theorem}[~\cite{GT21}]\label{th:Cn}
For any $p\ge 1$ and any $a,b$ such that $a\ge b\ge 1$, 

\[ \sep(C_{2p+1},a,b) = \left\{\begin{array}{ll}
a-b, & b\le a< 2b\\
b+(p+1)(a-2b), & 2b \le a \le 2b + \frac{b}{p}\\
a, &  a \ge 2b + \frac{b}{p}. 
                            \end{array}
\right.\]
\end{theorem}

Hence for $K_3=C_3$ we have:

\begin{corollary}\label{cor1}
\[\sep(K_3,a,b) = \left\{\begin{array}{ll}
a-b, & b\le a< 2b\\
2a-3b, & 2b \le a < 3b\\
a, &  a \ge 3b. 
                            \end{array}
\right.\]
\end{corollary}

For arbitrary values of $n$, we are (only) able to prove some bounds and two exact results for the remaining cases.

First, combining Propositions~\ref{propapb} and~\ref{biKn} allows to obtain the following bounds for the separation number when $2b\le a\le nb$ (note the (roughly) factor two between the lower and upper bound):
\begin{proposition}
For integers $a,b,n,x$ such that $a=xb$ is a multiple of both $\binom{n-1}{x-1}$ and $\lfloor \frac{a^2}{2b(n-1)}\rfloor$, we have
\[\left\lfloor a\frac{x}{2(n-1)}\right\rfloor\le\sep(K_n,a,b)\le a\frac{x-1}{n-1}.\]
\end{proposition}

\begin{proof}
Proposition~\ref{propapb} gives $\sep(K_n,a,b)\le a\frac{x-1}{n-1}$ and Proposition~\ref{biKn} with $x=a/b$ gives the lower bound.
\end{proof}

\begin{proposition}\label{propb2b}
For any $n\ge 3$ and , $a,b$ such that $b\le a \le 2b$, we have $\sep(K_n,a,b)= \lfloor\frac{2(a-b)}{n-1}\rfloor$.
\end{proposition}

\begin{proof}
 Let $a,b$ such that $b\le a \le 2b$ and let $c=\lfloor\frac{2(a-b)}{n-1}\rfloor$.
 Consider a $c$-separating $a$-list assignment $L$ of $K_n$, for $n\ge 3$.
 
From the hypothesis, we have $(n-1)c\le 2(a-b) < a$.

By the separation condition, the amplitude of $L$ on vertices from $S\subset [n]$, $|S|=i$, satisfies \[\Sigma_{S}(L)\ge a + (a-c)+\ldots + (a-(i-1)c) = ia- \frac12 i(i-1)c.\]
Therefore, as $c\le 2(a-b)/(n-1)$, we obtain 
\[\Sigma_{S}(L)\ge ia-\frac12 i(i-1) \frac{2(a-b)}{n-1} = ia \frac{n-i}{n-1} + ib\frac{i-1}{n-1}\ge ib.\] 
 Hence the amplitude condition is satified for any $i$, $1\le i\le n$ and thus $K_n$ is $(L,b)$-colorable.
 
For proving the upper bound, we construct counter-examples to show that, for $c'=c+1$, there exists  $c'$-separating a $a$-list assignment $L'$ that do not satisfy the amplitude condition, hence for which no $(L',b)$-coloring exists. 
 
 First, if $a\ge (n-1)c'$, then $L'$ is the quasi-symetrical list assignment constructed by setting for any $i\in [n]$, $|\Ip(\{i\})|= a-(n-1)c'$ and for any $i,j\in[n], i\ne j$, $|\Ip(\{i,j\})|= c'$, the other proper intersections being empty.
 
Then, by Equation~\ref{eq1} and Corollary~\ref{corAmp}, we have
\[\Sigma_{[n]}(L')=\sum _{i\in[n]}|\Ip(\{i\})| + \sum_{i,j\in[n], i<j}|\Ip(\{i,j\})|\\
= na -n(n-1)c' + n\frac{n-1}{2}c' = na - n\frac{n-1}{2}c'.\]

As $c' = c+1 > 2(a-b)/(n-1)$, we obtain
\[\Sigma_{[n]}(L') < na - n\frac{n-1}{2}\frac{2(a-b)}{n-1} = nb.\]
Hence the amplitude condition is not satisfied and $K_n$ is not $(L',b)$-colorable.

Now, if $(n-1)c \le a < (n-1)c'$, then let $\alpha$ such that $a=(n-1)c'-\alpha$, with $1\le \alpha\le n-1$.
Then $L'$ is the quasi-symetrical list assignment constructed by setting for any $i\in [n]$, $|\Ip(\{i\})|= a-(n-1)c+\alpha$ and for any $i,j\in[n], i\ne j$, $|\Ip(\{i,j\})|= c'$ if $i< j\le i+\alpha$ and $|\Ip(\{i,j\})|= c'-1$ otherwise, the other proper intersections being empty.

It is easy to observe that each vertex has a list of $a-(n-1)c'+\alpha  + \alpha(c'-1) + (n-1-\alpha)c' = a$ colors and that the list assignment is $c'$-separating. But, the amplitude of the full list is
\[\Sigma_{[n]}(L')=\sum _{i\in[n]}|\Ip(\{i\})| + \sum_{i,j\in[n], i<j}|\Ip(\{i,j\})|\\
= na -n(n-1)c' +n\alpha +\frac n2 (c'-1)\alpha + \frac n2 (n-1-\alpha) c' \]
\[ = na - \frac n2 (n-1)c' -\frac n2\alpha.\]

As $c' = c+1 > 2(a-b)/(n-1)$, we obtain
\[\Sigma_{[n]}(L') < na - \frac n2 (n-1)\frac{2(a-b)}{n-1} -\frac n2\alpha = nb -\frac n2 \alpha.\]
Hence the amplitude condition is not satisfied and $K_n$ is not $(L',b)$-colorable.

\end{proof}

\begin{proposition}\label{propn-1bnb}
For any $n\ge 3$ and , $a,b$ such that $(n-1)b\le a \le nb$, we have $\sep(K_n,a,b)= 2a-nb$.
\end{proposition}

\begin{proof}
 Let $a,b$ such that $(n-1)b\le a \le nb$ and  $c= 2a-nb$.
 Consider a $c$-separating $a$-list assignment $L$ of $K_n$, for $n\ge 3$. 
 
By the separation condition, the amplitude of $L$ on vertices from $S\subset [n]$, $|S|=2$, satisfies \[\Sigma_{S}(L)\ge a + (a-c)=2a-c=2a-2a+nb=nb.\]

Hence, for any $S\subset [n]$, $|S|\ge 1$, $\Sigma_{S}(L)\ge |S|b$ and the amplitude condition is satisfied.

Now, we construct a counter example for $c'=c+1= 2a-nb+1$ of a $c'$-separating $a$-list assignment $L'$ for which no $(L,b)$-coloring exists. Let $L'$ be constructed by setting for any $i\in [n]$, $|\Ip([n]\setminus \{i\})|= a-c'$ and  $|\Ip([n])|= (n-1)c'-(n-2)a$, the other proper intersections being empty. Note that $(n-1)c'-(n-2)a = (n-1)(2a-nb+1) -(n-2)a = a(2n-2-n+2) -(n-1)(nb-1) = n(a-(n-1)b)+n-1 \ge 0$ since $a\ge (n-1)b$ by the hypothesis.
Then, by Equation~\ref{eq1} and Corollary~\ref{corAmp}, we have
\[\Sigma_{[n]}(L')=n(a-c')+(n-1)c'-(n-2)a = 2a-c'=2a-2a+nb -1 = nb-1 < nb.\]

Thus the amplitude condition is not satisfied and $K_n$ is not $(L',b)$-colorable.

\end{proof}

Remark that Corollary~\ref{cor1} can also be deduced from Propositions~\ref{propb2b} and~\ref{propn-1bnb}.

Putting all the partial results together and the computations made, we propose the following conjecture:

\begin{conjecture}
for any $n\ge 4,a,b, p$ with $2 \le p\le n-2$ and $pb\le a < (p+1)b$, we have \[\sep(K_n,a,b)=\left\lceil \frac{2p a -p(p+1) b}{n-1}\right\rceil+\epsilon,\] with $\epsilon\in\{-1,0\}$.
\end{conjecture}

Let us explain how we arrive to this conjecture: First, as the degree of every vertex in $K_n$ is $n-1$, then, for a balanced distribution of colors among vertices, they should be grouped by packets of $n-1$. Hence the separation number must be (the ceiling of) a certain function of $a,b,p$ divided by $n-1$. By Property~\ref{propp1}, this function must also be close to an affine function. Second, supported by Proposition~\ref{propapb}, we conjecture that when $a=pb$, then $\sep(K_n,a,b)=a(p-1)/(n-1)$. All this together lead us to propose the above conjecture, with $\epsilon$ being the correcting term depending on $a,b,n$.

In particular, for $n=4$, the following refinment is conjectured, where only the case $2b \le a < 3b$ remains to be verified:

\begin{conjecture}
\[\sep(K_4,a,b) = \left\{\begin{array}{ll}
\lfloor\frac{2(a-b)}{3}\rfloor, & b\le a< 2b\\
\lceil\frac{4a-6b-1}{3}\rceil, & 2b \le a < 3b\\
2a-4b, & 3b \le a < 4b\\
a, &  a \ge 4b. 
                            \end{array}
\right.\]
\end{conjecture}

\appendix

\section{Vectorial solution of the ILP problem for symetrical list assignments}\label{app}
In this section, we show how the counter examples can be found using an algebraic formulation.

Let $\mathbf{e}_1, \ldots, \mathbf{e}_n$ be a canonical base of the vectorial space of dimension $n$.

We define the vector  
\[\mathbf{a}=\sum_{i=1}^{n} \dbinom{n-1}{i-1}\mathbf{e}_i.\]


We can remark that $\mathbf{a}$ is symetrical.
For a vector $\mathbf{x}\in \mathbb{R}^n$, we define the linear application $\varphi_a(\mathbf{x})=\mathbf{a}\times\mathbf{x}$. 

We also define the antisymetrical vectors $\forall i\in[1,\lfloor\frac n2\rfloor]$,

$\mathbf{as}(i) = \mathbf{e_i}-\mathbf{e}_{n+1-i}$ and, $\forall i\in[3,\lceil\frac n2\rceil]$, the Pascal's triangle vectors 

$\mathbf{tp}(i) = \mathbf{e}_i+\mathbf{e}_{i-1}-\binom{n}{i-1}\mathbf{e}_{1}$.

Lastly, we define the binomial vector $\mathbf{bn} = \sum_{i=1}^n \mathbf{e}_i - 2 ^{n-1} \mathbf{e}_1$.

\begin{lemma}
 The kernel $\ker{(\varphi_a)}$ has $\{ \{\mathbf{as}(i)\}, \{\mathbf{tp}(i)\}, \mathbf{bn}\}$ as basis and is thus of dimension $n-1$.
\end{lemma}
\begin{proof}
 By construction, it is easy to observe that $F_a= \{ \{\mathbf{as}(i)\}, \{\mathbf{tp}_i\}, \mathbf{bn}\}\subset \ker{(\varphi_a)}$. As the number of elements of $F_a$ is equal to the dimension of $\ker{(\varphi_a)}$, it suffices to show that $F_a$ is a family of free vectors. We consider two cases depending on the parity of $n$.
 
 \begin{itemize}
  \item Case $n$ even. Let $\alpha_i$, $1\le i\le \frac n2$ and $\beta_j$, $3\le j\le \frac n2$ and $\gamma\in \mathbb{R}$ such that 
  
  \[\sum_{i=1}^{n/2} \alpha_i \mathbf{as}(i) + \sum_{j=3}^{n/2} \beta_j \mathbf{tp}(j) + \gamma \mathbf{bn} = \mathbf{0}.\]
  
  For the $k$-th coordinate of the left side of the above equality, with $k\in[\frac n2 +1, n]$, we have  $-\alpha_{n+1-k} + \gamma = 0$, thus $\gamma = \alpha_1=\alpha_2=\ldots = \alpha_{n/2}$.
  
  For $k=2$, $\alpha_2 + \beta_3 + \gamma = 0 \Rightarrow \beta_3 =-2\gamma$.
  
  For $k\in[3,\frac n2 -1]$, $\alpha_{k} + \beta_k + \beta_{k+1}+\gamma = 0 \Rightarrow \beta_4=\ldots = \beta_{n/2} = 0$.
  
  For $k=n/2$, $\alpha_k + \beta_k+\gamma = 0 \Rightarrow \alpha_{n/2} = -\gamma$. But we have seen that $\alpha_{n/2}=\gamma$, hence $\gamma=0$.

  
  Consequently, we obtain that all coefficients are equal to zero.
  
  \item Case $n$ odd.
  Let $\alpha_i$, $1\le i\le \frac{n-1}{2}$ and $\beta_j$, $3\le j\le \frac{n+1}{2}$ and $\gamma\in \mathbb{R}$ such that 
  
  \[\sum_{i=2}^{(n-1)/2} \alpha_i \mathbf{as}(i) + \sum_{j=3}^{(n+1)/2} \beta_j \mathbf{tp}(j) + \gamma \mathbf{bn} = \mathbf{0}.\]
  
  For the $k-th$ coordinate of the left side of the above equality, $k\in[\frac{n+3}{2}, n-1]$, we have $-\alpha_{n+1-k} + \gamma = 0$. Thus $\gamma = \alpha_1=\ldots = \alpha_{(n-1)/2}$.
  
  For $k=(n+1)/2$, $\beta_k+\gamma = 0 \Rightarrow \beta_{(n+1)/2} = -\gamma$.
  
  For $k\in[3,\frac{n-1}{2} -1]$, $\alpha_{k} + \beta_k + \beta_{k+1}+\gamma = 0$. Hence $\beta_k = -\alpha_k = -\gamma$.

  Finally, for $k=2$, $\alpha_2 + \beta_3 + \gamma = 0 \Rightarrow \gamma=0$.
 \end{itemize}
\end{proof}

Similarly, for the separation condition, we define the vector \[\mathbf{c}=\sum_{i=2}^n \binom{n-2}{i-2} \mathbf{e}_i\] and the linear application $\varphi_c(\mathbf{x}) = \mathbf{c}\times\mathbf{x}$. 

We also define the vectors 
\begin{itemize}
 \item $\forall i\in[2,\lfloor\frac n2\rfloor]$,
$\mathbf{asc}(i) = \Delta(i,n)\mathbf{as}(1) + \mathbf{as}(i)$,\\ with $\Delta(i,n)=\binom{n-2}{i-2} -\binom{n-2}{n-1-i} = \binom{n-2}{i-2} -\binom{n-2}{i-1}$.
\item $\forall i\in[3,\lceil\frac n2\rceil]$, $\mathbf{tpc}(i) = \binom{n-1}{i-2}\mathbf{as}(1) + \mathbf{tp}(i)$.
\item $\mathbf{bnc} = 2^{n-2}\mathbf{as}(1) + \mathbf{bn}$.
\end{itemize}

\begin{lemma}
 The intersection of kernels $\ker{(\varphi_a)}\cap \ker{(\varphi_c)}$ has  $\{ \{\mathbf{asc}(i)\}, \{\mathbf{tpc}(i)\}, \mathbf{bnc}\}$ as a basis and is thus of dimension $n-2$.
\end{lemma}

\begin{proof}
 It is easy to see that $F_{ac}=\{ \{\mathbf{asc}(i)\}, \{\mathbf{tpc}(i)\}, \mathbf{bnc}\} \subset \ker{(\varphi_a)}\cap \ker{(\varphi_c)}$ and has $n-2$ vectors. Since the dimension of $\ker{(\varphi_a)}\cap \ker{(\varphi_c)}$ is equal to $n-2$ too, it suffices to show that $F_{ac}$ is a free family.
 We consider two cases depending on the parity of $n$.
 
 \begin{itemize}
  \item Case $n$ even. Let $\alpha_i$, $2\le i\le \frac n2$ and $\beta_j$, $3\le j\le \frac n2$ and $\gamma\in \mathbb{R}$ such that 
  
  \[\sum_{i=2}^{n/2} \alpha_i \mathbf{asc}(i) + \sum_{j=3}^{n/2} \beta_j \mathbf{tpc}(j) + \gamma \mathbf{bnc} = \mathbf{0}.\]
  
  For $k\in[\frac n2 +1, n-1]$, $-\alpha_{n+1-k} + \gamma = 0 \Rightarrow \gamma = \alpha_2=\alpha_3=\ldots = \alpha_{n/2 -1}=\alpha_{n/2}$.
  
  For $k=n/2$, $\alpha_k + \beta_k+\gamma = 0 \Rightarrow \beta_{n/2} = -2\gamma$.
  
  For $k\in[3,\frac n2 -1]$, $\alpha_{k} + \beta_k + \beta_{k+1}+\gamma = 0 \Rightarrow \beta_3 = \ldots = \beta_{n/2 -1} = 0$.
  
  Finally, for $k=2$, $\alpha_2 + \beta_3 + \gamma = 0 \Rightarrow 2\gamma=0 \Rightarrow \gamma=0$.
  
  \item Case $n$ odd.
  Let $\alpha_i$, $2\le i\le \frac{n+1}{2}$ and $\beta_j$, $3\le j\le \frac{n+1}{2}$ and $\gamma\in \mathbb{R}$ such that 
  
  \[\sum_{i=2}^{(n+1)/2} \alpha_i \mathbf{asc}(i) + \sum_{j=3}^{(n+1)/2} \beta_j \mathbf{tpc}(j) + \gamma \mathbf{bnc} = \mathbf{0}.\]
  
  For $k\in[\frac{n+3}{2}, n-1]$, $-\alpha_{n+1-k} + \gamma = 0 \Rightarrow \gamma = \alpha_2=\alpha_3=\ldots = \alpha_{(n-1)/2}$.
  
  For $k=(n+1)/2$, $-\alpha_{(n-1)/2}+ \beta_k+\gamma = 0 \Rightarrow \beta_{(n+1)/2} = 0$.
  
  For $k\in[3,\frac{n-1}{2} -1]$, $\alpha_{k} + \beta_k + \beta_{k+1}\gamma = 0 \Rightarrow \beta_{(n-1)/2} = 2\gamma, \beta_{(n-3)/2}=0, \beta_{(n-5)/2}=2\gamma,\ldots$.
  
  Finally, for $k=2$, $\alpha_2 + \beta_3 + \gamma = 0 \Rightarrow \beta_3 = -2\gamma$.
  If $(n+1)/2 \equiv 3 \bmod{4}$ then $\gamma=0$. Otherwise ($(n+1)/2 \not \equiv 3 \bmod{4}$), then $2\gamma=-2\gamma$ and thus $\gamma=0$.
  
 \end{itemize}
In both cases, we have shown that all coefficients are equal to zero, proving that the set of vectors forms a free family.
\end{proof}

We define $\mathbf{\psi} = \sum_{i=1}^n \binom{n}{i} \mathbf{e}_i$ which is the measure of the amplitude and $E(a,c)=\{\mathbf{x} \text{ such that } \varphi_a (\mathbf{x})=a \text{ and } \varphi_c(\mathbf{x})=c\}$.

Now, using the maximum principle, we define the two optimal solution vectors $\mathbf{x_i}$, $i=1, 2$. 
\begin{itemize}
 \item If $a\ge (n-1)c$, we define $\mathbf{x_1} = (a-(n-1)c)\mathbf{e}_1 + c\mathbf{e}_2$.

We can observe that $\mathbf{x}_1\in E(a,c)$. In fact $E(a,c) = \mathbf{x}_1 + ker(\varphi_a) \cap \ker({\varphi_c})$. This solution corresponds to the counter-example used in proof of Proposition~\ref{propb2b}.

\item If $ a\le \frac{n-1}{n-2}c$, $\mathbf{x_2} = (a-c)\mathbf{e}_{n-1} + ((n-1)c-(n-2)a)\mathbf{e}_n$.

We can observe that $\mathbf{x}_2\in E(a,c)$. In fact $E(a,c) = \mathbf{x}_2 + ker(\varphi_a) \cap \ker({\varphi_c})$. This solution corresponds to the counter-example used in proof of Proposition~\ref{propn-1bnb}.

\end{itemize}

\end{document}